\documentclass[12 pt]{article}

\usepackage{graphics}
\usepackage{latexsym}
\usepackage{amsfonts}
\usepackage{amsmath}
\usepackage{amsthm}
\usepackage{multicol}
\usepackage{lscape}
\theoremstyle{plain}
\newtheorem{thm}{Theorem}
\newtheorem{lemma}[thm]{Lemma}
\newtheorem{prop}[thm]{Proposition}

\newtheorem{claim}{Claim}

\newcommand{\N}{\mathbb{N}}
\newcommand{\Z}{\mathbb{Z}}

\begin{document}

\title{$f$-vectors of Simplicial Posets that are Balls}
\author{Samuel Kolins\\
Department of Mathematics,\\
Cornell University, Ithaca NY, 14853-4201, USA,\\
skolins@math.cornell.edu}

\maketitle

\begin{abstract}
Results of R. Stanley and M. Masuda completely characterize the $h$-vectors of simplicial posets whose order complexes are spheres.
In this paper we examine the corresponding question in the case where the order complex is a ball.
Using the face rings of these posets, we develop a series of new conditions on their $h$-vectors.
We also present new methods for constructing poset balls with specific $h$-vectors.
These results allow us to give a complete characterization of the $h$-vectors of simplicial poset balls up through dimension six.
\end{abstract}

\section{Introduction}

A \emph{simplicial poset} $P$ is a finite poset containing a minimal element $\hat{0}$ such that for every $p \in P$ the closed interval $[\hat{0},p]$ is a Boolean algebra.
For any simplicial poset $P$ there is a regular CW-complex $\Gamma(P)$ such that $P$ is the face poset of $\Gamma(P)$ (see  \cite{Bjorner84}).
The closed faces of $\Gamma(P)$ are simplexes, but two faces can intersect in a subcomplex of their boundaries instead of just a single face.
In particular, $\Gamma(P)$ can have multiple faces on the same vertex set.
Throughout this paper we will identify each closed face of $\Gamma(P)$ with the corresponding element of the poset $P$.

Let $\overline{P} = P - \{\hat{0} \}$.
The \emph{order complex} of $\overline{P}$, denoted $\Delta(\overline{P})$ is the simplicial complex whose vertices are the elements of $\overline{P}$ and whose faces are the chains of $\overline{P}$.
For a simplicial poset $P$, when we refer to the order complex of $P$ we mean $\Delta(\overline{P})$.
For any CW-complex $K$ let $|K|$ denote the geometric realization of $K$.
Then $|\Gamma(P)|$ and $|\Delta(\overline{P})|$ are homeomorphic (in fact, $\Delta(\overline{P})$ is isomorphic to the barycentric subdivision of $\Gamma(P)$).
Therefore, in the following we will often refer to topological properties of $|\Gamma(P)|$ or $|\Delta(\overline{P})|$ as being properties of the poset $P$.
In this paper we study simplicial posets that are balls.

The \emph{$i$th face number} of $\Gamma(P)$, denoted $f_i (\Gamma(P))$, is the number of $i$-dimensional faces of $\Gamma (P)$.
Equivalently, $f_i (\Gamma(P))$ is the number of elements $p \in P$ such that $[\hat{0},p]$ is a Boolean algebra of rank $i+1$.
In particular, $f_{-1}(\Gamma(P)) = 1$ corresponding to the empty set in $\Gamma(P)$ or the element $\hat{0}$ in $P$.
The $\emph{dimension}$ of $\Gamma(P)$ is the largest $i$ such that $f_{i}(\Gamma(P))$ is non-zero.
We define $f_i(P)$, the $i$th face number of the poset $P$, by $f_i(P) = f_i (\Gamma(P))$.
If the poset $P$ is clear from the context we often just write $f_i$ instead of $f_i(P)$ or $f_i (\Gamma(P))$.

Let $d-1$ be the dimension of $P$.
We record all of the face numbers of $P$ in a single vector $f(P) = (f_{-1}, f_0, f_1, \ldots , f_{d-1})$ called the \emph{f-vector} of $P$.
It will often be easier to work with an equivalent encoding of the face numbers called the \emph{h-vector}.
The entries $(h_0, h_1, \ldots h_d)$ of the h-vector are obtained from the face numbers by the relation
\begin{equation*}
\sum_{i=0}^d h_i x^i = \sum_{i=0}^{d} f_{i-1} x^i (1-x)^{d-i}.
\end{equation*}
In the case where $\Gamma(P)$ is a simplicial complex, this definition of the $h$-vector of $\Gamma(P)$ agrees with the standard definition of the $h$-vector of a simplicial complex.
Note that if $P$ is a simplicial poset of dimension $d-1$ then $h_d = \sum_{i=0}^d (-1)^{d-i} f_{i-1} = (-1)^{d-1} \tilde\chi (|\Gamma(P)|)$.
In particular, if $P$ is a $(d-1)$-ball then $h_d = 0$.
In many cases we will study the differences between consecutive entries of the $h$-vector.
We therefore define $g_i = h_i - h_{i-1}$, with $g_0 := h_0 = 1$.

A significant area of study is attempting to characterize the possible $f$-vectors for various types of simplicial posets.
Complete characterizations are already known for Cohen-Macaulay posets (see section \ref{S:Background}) and spheres.
For simplicial posets that are spheres, sufficiency was proved by Stanley \cite[Theorem 4.3, Remark 4]{Stanley91} and necessity by Masuda \cite[Corollary 1.2]{Masuda}.
\begin{thm}{\label{T:Spheres}}
Let $\mathbf{h}= (h_0,h_1, \ldots ,h_d) \in \Z^{d+1}$.
Then there exists a simplicial poset $P$ with $h(P) = \mathbf{h}$ and $|\Gamma(P)|$ a sphere if and only if 
$h_0 =1$,
$h_i = h_{d-i}$ for $0 \leq i \leq d$, and either $h_i > 0$ for $0 \leq i \leq d$ or $\sum_{i=0}^d h_i$ is even.
\end{thm}
The condition $h_i = h_{d-i}$ means that when $d$ is odd $\sum_{i=0}^d h_i$ is always even.
In the case where $d$ is even the parity of $\sum_{i=0}^d h_i$ is equal to the parity of $h_{d/2}$.

In this paper we investigate the question of characterizing the $h$-vectors of posets $P$ such that $|\Gamma(P)|$ is a ball.
In section \ref{S:Boundary} we relate the $h$-vector of a poset that is a ball to the $h$-vector of its boundary poset, which is a sphere.
This will allow us to translate the known conditions on the $h$-vectors of posets that are spheres to conditions on our ball.
However, as in the case of simplicial \emph{complexes} that are balls, this is not sufficient to completely characterize the $h$-vectors of balls \cite{KolinsBall}.

In sections \ref{S:Monotonicity} and \ref{S:Parity} we develop additional necessary conditions on the $h$-vectors of balls.
In section \ref{S:Monotonicity} we show that when $h_1$ of the boundary sphere is zero there is a surjective map from the face ring of the ball modulo a linear system of parameters to the face ring of the boundary sphere modulo a related linear system of parameters.
This results in a series of new inequalities relating the entries of the $h$-vectors of the ball and the boundary sphere.
Section \ref{S:Parity} gives a number of conditions on the $h$-vector of a ball that force the sum of the entries of the $h$-vector to be even.
All of these conditions require that some entry of $h$-vector of the ball is zero (besides $h_d$, which is always zero) and that some entry of the $h$-vector of the boundary sphere is also zero.
The first two results in this section follow from counting arguments on the facets of the ball.
The other two conditions come from adding to our ball the cone over the boundary the ball, resulting in a sphere of the same dimension as the original ball.
We then look at the restriction map from the face ring of this new sphere to the face ring our original ball and use this map to transfer known conditions on the sphere to conditions on the ball.

In section \ref{S:Constructions} we present constructions for obtaining simplicial posets that are balls and have specific $h$-vectors.
These constructions all use the idea of shellings \cite[Definition 4.1]{Bjorner84} to ensure that the resulting complex is a ball and that it has the desired $h$-vector.
We present two general constructions as well as a third result that yields additional $h$-vectors in dimension five.
We conclude the paper in section \ref{S:Summary} by using the previous results to give a complete characterization of the $h$-vectors of simplicial posets that are balls up through dimension 6.
We also discuss what types of problems remain in the higher dimensional cases.

\section{Notation and Background}{\label{S:Background}}

In this section we provide background on many of the ideas mentioned in the introduction,
as well as some additional useful results of Masuda about simplicial poset spheres.

\subsection{Shellings of $\Gamma(P)$}
The \emph{facets} of any CW-complex are the maximal faces with respect inclusion.
A CW-complex is \emph{pure} if all of its facets have the same dimension.
Let $P$ be a simplicial poset of dimension $d-1$.
Then the facets of $\Gamma(P)$ correspond to the maximal elements of $P$.

Consider the case where $\Gamma(P)$ is a pure complex.
A \emph{shelling} of $\Gamma (P)$ is an ordering $F_1, F_2, \ldots , F_t$ of the (closed) facets of $\Gamma (P)$ such that $F_j \cap ( \cup_{i=1}^{j-1} F_i)$ is a union of (closed) \emph{facets} of $\partial F_j$. 
This is equivalent to the definition of a shelling of a CW-complex used by Bj{\"o}rner \cite[Definition 4.1]{Bjorner84} specialized to the case of $\Gamma(P)$ for a simplicial poset $P$.
Define the \emph{restriction face} of $F_k$, denoted $\sigma(F_k)$, to be the set of vertices $v$ of $F_k$ such that the facet of $\partial F_k$ not containing $v$ is in $\cup_{i=1}^{k-1} F_i$.
Then the entries of the $h$-vector of $P$ are given by $h_j = |\{F_k : |\sigma (F_k)| = j \}|$.
We also use $\sigma(F_k)$ to refer to the face of $F_k$ containing exactly the vertices in the set $\sigma(F_k)$.

\subsection{Cones of Posets}

Given a simplicial poset $P$, the \emph{cone} over $P$ is the simplicial poset $P \times [1,2]$ where $[1,2]$ is the poset of two elements with $2 > 1$.
More specifically, the elements of $P \times [1,2]$ are the ordered pairs $(p,i)$ where $p \in P$ and $i \in \{1,2\}$ and the covering relations are:
\begin{itemize}
 \item If $p$ covers $q$ in $P$, then $(p,i)$ covers $(q,i)$ in $P \times [1,2]$ for $i \in \{1,2\}$.
 \item For all $p \in P$, $(p,2)$ covers $(p,1)$ in $P \times [1,2]$.
\end{itemize}
Topologically, $|\Gamma(P \times [1,2])|$ is the cone over the complex $|\Gamma(P)|$.
The dimension of $P \times [1,2] $ is one greater than that of $P$.
A straightforward calculation shows that the $h$-vector of $P \times [1,2]$ is the same as that of $P$ except augmented by $h_{d+1} = 0$.

\subsection{The Face Ring of a Simplicial Poset}
We now describe Stanley's idea of the face ring of a simplicial poset \cite{Stanley91}.
Let $k$ be an infinite field.
Define $S := k[x_p : p \in \overline{P}]$ to be the polynomial ring over $k$ with variables indexed by the elements of $\overline{P}$.
We define a grading on $S$ by letting the degree of $x_p$ be one more than the dimension of the face in $\Gamma(P)$ corresponding to $p$ (so $[\hat{0},p]$ is a Boolean algebra of rank equal to the degree of $x_p$).
For elements $p$ and $q$ in $P$ the \emph{meet} of $p$ and $q$, denoted $p \wedge q$, is the largest element that is less than both $p$ and $q$.
Since $P$ is a simplicial poset we know that the element $p \wedge q$ is well defined whenever $p$ and $q$ have a common upper bound.
Define $I_P$ to be the ideal of $S$ generated by all elements of the form $x_p x_q - {x_{p \wedge q}} \sum_r x_r$ where the sum is over all minimal upper bounds $r$ of $p$ and $q$.
In the case where $p$ and $q$ have no common upper bound in $P$, this reduces to the element $x_p x_q$.
Define the \emph{face ring} of $P$ to be $A_P := S / I_P$.

\subsection{Cohen-Macaulay Simplicial Posets}
A simplicial poset $P$ is \emph{Cohen-Macaulay} if its order complex $\Delta(\overline{P})$ is a Cohen-Macaulay simplicial complex.
For a simplicial complex $\Delta$, Munkres \cite{Munkres} proved that $\Delta$ is Cohen-Macaulay if and only if for all points $p \in |\Delta|$ and all $i < \dim \Delta$, $\tilde{H}_i(|\Delta|;k) = H_i(|\Delta|,|\Delta|-p;k)=0$.
In particular, whenever $|\Delta(\overline{P})| \cong |\Gamma(P)|$ is a ball or sphere $P$ is Cohen-Macaulay.
Stanley proved that a simplicial poset $P$ is Cohen-Macaulay if and only if its face ring $A_P$ is a Cohen-Macaulay ring \cite[Corollary 3.7]{Stanley91}.
Using this result, Stanley \cite[Theorem 3.10]{Stanley91} proved that if $Q$ is a Cohen-Macaulay simplicial poset then $h_0(Q) = 1$ and $h_i(Q) \geq 0$ for $i \geq 1$ (he also proved that these are sufficient conditions to characterize Cohen-Macaulay simplicial posets).
Stanley's book \cite{StanleyCCA} is a good reference for more information on Cohen-Macaulay rings and complexes.

Let $T = T_0 \oplus T_1 \oplus \cdots$ be any finitely generated graded algebra over the (infinite) field $k = T_0$.
The \emph{Hilbert function} of $T$ is defined by $F(T,i) :=  \dim_k T_i$ where $i \geq 0$.
Let $d = \dim T$.
Then a linear system of parameters (l.s.o.p) for $T$ is a collection of elements $\theta_1, \ldots , \theta_d \in T_1$ such that $T$ is finitely generated as a $k[\theta_1, \ldots , \theta_d]$-module.
From \cite[Theorem 3.10]{Stanley91} we know that when $P$ is a Cohen-Macaulay simplicial poset $\dim A_P = \dim (\Gamma(P)) +1$ and a l.s.o.p. for $A_P$ exists.
Further, by \cite[Theorem 3.8]{Stanley91} when $P$ is a Cohen-Macaulay simplicial poset and $\theta_1, \ldots , \theta_d$ is a l.s.o.p. for $A_P$ we have $F(A_P/(\theta_1, \ldots , \theta_d),i) = h_i(P)$.

\subsection{Additional Results about Simplicial Poset Spheres}
As mentioned in the introduction, Theorem \ref{T:Spheres} gives a complete characterization of the possible $h$-vectors of simplicial posets that are spheres.
In addition to the numerical result we will often need some of the stronger statements that were used to prove the necessity of the claim.
The following theorem is discussed on pages 343-344 of the original proof of necessity due to Masuda \cite{Masuda} and is Theorem 2 in the paper \cite{MillerReiner} published two years later by Miller and Reiner giving a simplified proof of Masuda's result.

\begin{thm}{\label{T:SphereZero}}
Let $P$ be a simplicial poset such that $|\Gamma(P)|$ is a $(d-1)$-sphere.
If $h_i(P) = 0$ for some $i$ strictly between zero and $d$, then for every subset $V = \{v_1, \ldots , v_d\}$ of the vertices of $\Gamma(P)$ there is an even number of facets of $\Gamma(P)$ with vertex set $V$.
\end{thm}

One consequence of the proof of Theorem \ref{T:SphereZero} is the following result
that relates the parity of the number of facets on a vertex set to the product of those vertices in the algebra $A_P$.

\begin{prop}{\label{T:SphereZeroAlg}}
Let $P$ be a simplicial poset such that $|\Gamma(P)|$ is a $(d-1)$-sphere and let $V = \{v_1, \ldots , v_d\}$ be a subset of the vertices of $\Gamma(P)$.
Let $\Theta$ be a l.s.o.p. for $A_P$.
If $x_{v_1} \cdots x_{v_d}$ is zero in $A_P / \Theta$ then there are an even number of facets of $\Gamma(P)$ with vertex set $V$.
\end{prop}

\section{The $h$-vector of the Boundary of a Simplicial Poset}{\label{S:Boundary}}

The goal of this section is to relate the $h$-vector of a simplicial poset whose geometric realization is a manifold with boundary to the $h$-vector of the boundary complex.
In the case of balls, this will allow us to use Theorem \ref{T:Spheres} about the $h$-vectors of posets that are spheres to restrict the possible $h$-vectors of balls.

Our starting point is a paper by I.G. Macdonald \cite[Theorem 2.1]{Macdonald71} that gives the desired relationship for the case of simplicial complexes whose geometric realizations are manifolds with boundary.
The entire first section of Macdonald's paper is done in the generality of cell complexes and applies in our case.
When Macdonald specializes to the case of simplicial complexes to prove Theorem 2.1, the only property of simplicial complexes that he uses is the fact that for each simplex $y$ in the complex the interval $[\hat{0},y]$ in the face poset is a Boolean algebra.
Since this fact is true for simplicial posets, Macdonald's result holds in this more general setting as well.
Expressing his result in terms of $h$- and $g$-vectors we have the following theorem.

\begin{thm}{\label{GDS}}
Let $P$ be a $(d-1)$-dimensional simplicial poset such that $|\Gamma(P)|$ is a manifold with boundary.
Then
\[ h_{d-i}(P) - h_i(P) = \binom{d}{i} (-1)^{d-1-i} \tilde\chi (|\Gamma(P)|) - g_i(\partial \Gamma(P)) \]
for all $0 \leq i \leq d$.
\end{thm}

In the special case where $|\Gamma(P)|$ is a $(d-1)$-ball this reduces to the equation 
$h_i(P) - h_{d-i}(P) = g_i(\partial \Gamma(P))$.
In particular, for $0 \leq j \leq d$ we have
\begin{equation}{\label{E:Boundary}}
h_j(\partial \Gamma(P)) = \sum_{i=0}^j (h_i(P) - h_{d-i}(P)).
\end{equation}
Note that in the case where $d$ is odd, $\sum_{i=0}^d h_i(P) = h_{(d-1)/2}(\partial \Gamma(P))$, which has the same parity as $\sum_{i=0}^{d-1} h_i (\partial \Gamma(P))$.

With this relationship we can now give a first set of necessary conditions on the $h$-vectors of simplicial poset balls.
As discussed in section \ref{S:Background}, any simplicial poset $P$ such that $|\Gamma(P)|$ is ball is a Cohen-Macaulay poset.
Combining Stanley's characterization of the $h$-vectors of Cohen-Macaulay posets with Theorem \ref{T:Spheres} and Equation (\ref{E:Boundary}) we have the following conditions.

\begin{thm}{\label{T:BasicCond}}
Let $P$ be a $(d-1)$-dimensional simplicial poset such that $|\Gamma(P)|$ is a ball.
Then $h_0 (P) = 1$, $h_d (P) = 0$, $h_i(P) \geq 0$ for all $i$, and $\sum_{i=0}^j (h_i(P) - h_{d-i}(P)) \geq 0$ for $0 \leq j \leq \lfloor (d-1)/2 \rfloor$.
Further, if $d$ is odd and $\sum_{i=0}^d h_i(P)$ is odd then $\sum_{i=0}^j (h_i(P) - h_{d-i}(P)) > 0$ for $0 \leq j \leq (d-1)/2$.
\end{thm}

\section{Inequalities Relating a Ball and its Boundary}{\label{S:Monotonicity}}
Consider a simplicial poset ball such that $h_1$ of the boundary sphere is zero.
In the following we derive inequalities relating the $h$-vectors of the boundary sphere and the ball in this case.
The idea of the argument follows that of a similar result by Stanley \cite[Theorem 2.1]{StanleyMonotone} for the $h$-vectors of simplicial complexes.

One of the main tools in this proof is a useful characterization of linear systems of parameters for the ring $A_P$.
Fix an ordering $\{v_1, \ldots, v_n\}$ of the vertices of $P$.
Let $\theta_1, \ldots , \theta_d$ be a collection of homogeneous degree one elements of $A_P$.
We can write each element of our collection as a linear combination of the $x_{v_j}$, $\theta_i = \sum_{j=1}^n \Theta_{i,j} x_{v_j}$.
This gives a $d \times n$ matrix $\Theta_{i,j}$ whose rows correspond to the $\theta_i$.

Let $F$ be a facet of $P$.
Define $\Theta_F$ to be the $d \times d$ submatrix of $\Theta_{i,j}$ obtained by restricting to the columns corresponding to the vertices of $F$.
Then we have the following characterization of which collections of degree one elements are linear systems of parameters.

\begin{lemma}{\label{L:lsop}}
Let $P$ be a simplicial poset and let $\theta_1, \ldots , \theta_d$ be a collection of homogeneous degree one elements of $A_P$.
Then $\theta_1, \ldots , \theta_d$ is a l.s.o.p. for $A_P$ if and only if $\det (\Theta_F) \neq 0$ for all facets $F$ of $P$.
\end{lemma}

The only if part of the lemma was proved by Masuda \cite[Lemma 3.1]{Masuda} and Miller and Reiner \cite[p. 1051]{MillerReiner}.
The if direction follows from Proposition 5 of Miller and Reiner's paper.
In this proposition Miller and Reiner show that for a l.s.o.p. $\theta_1, \ldots , \theta_d$, $A_P/(\theta_1, \ldots, \theta_d)$ is spanned $k$-linearly by the images of the $x_G$ for all elements $G \in P$.
The only property of the l.s.o.p. used in the proof is the non-zero determinant assumption in the above lemma.

Now let $P$ be a $(d-1)$-ball.
Then $\partial (\Gamma (P))$ is a $(d-2)$-sphere.
If $h_1(\partial (\Gamma (P)))=0$ we know that $\partial (\Gamma (P))$ has only $d-1$ vertices.
Therefore every facet of $\partial (\Gamma (P))$ has the same vertex set.
Let $F$ be a facet of $\Gamma (P)$ such that a codimension-one face of $F$ is in $\partial (\Gamma (P))$.
Let $v$ be the vertex of $F$ that is not $\partial (\Gamma(P))$.  Note that all of the vertices of $\Gamma(P)$ not in $F$ are interior vertices.

Let $\theta_1, \ldots, \theta_d$ be a l.s.o.p. for $A_P$.
By Lemma \ref{L:lsop} we know that $\Theta_F$ has non-zero determinant.
Thus the span of $\theta_1, \ldots, \theta_d$ contains some element $\theta' = x_v + \sum_{w \notin F} c_w x_w$ where the sum is over the vertices of $\Gamma(P)$ not in $F$ and the $c_w$ are constants in $k$.
This allows us to choose a new l.s.o.p. $\theta_1', \ldots, \theta_d'$ for $A_P$ such that $\theta_d'$ is a linear combination of interior vertices of $\Gamma (P)$.
By Lemma \ref{L:lsop} we know that $\det (\Theta'_G) \neq 0$ for all facets $G$ of $\Gamma(P)$.

Let $Q$ be the face poset of $\partial \Gamma (P)$.
Let $f : A_P \rightarrow A_Q$ be given by setting all variables corresponding to faces in $\Gamma(P) \backslash \partial \Gamma(P)$ equal to zero.
Identify the l.s.o.p. $\theta_1', \ldots, \theta_d'$ with its image under $f$ in $A_Q$.
Let $H$ be any facet of $\partial \Gamma(P)$. 
The last row of $\Theta'_H$ is all zeros, so the $(d-1) \times (d-1)$ minor given by the first $(d-1)$ rows must have a non-zero determinant.
Again using Lemma \ref{L:lsop} we have that $\theta_1', \ldots, \theta_{d-1}'$ is a linear system of parameters for $\partial \Gamma (P)$.

Therefore, $f$ induces a degree preserving surjection 
\[ f: A_P/(\theta_1', \ldots , \theta_d') \rightarrow A_Q/(\theta_1', \ldots , \theta_{d-1}'). \]
Hence $h_i (P) = F(A_P/(\theta_1', \ldots , \theta_d'),i)  \geq F(A_Q/(\theta_1', \ldots , \theta_{d-1}'),i) = h_i (Q)$.
Summarizing, we have proved the following theorem.

\begin{thm}{\label{T:Monotone}}
Let $P$ be a $(d-1)$-dimensional simplicial poset such that $|\Gamma(P)|$ is a ball and $h_1(\partial (\Gamma (P))) = 0$.
Then $h_i (P) \geq h_i(\partial (\Gamma (P)))$ for all $i \geq 0$.
\end{thm}

\section{Parity conditions on the sum of the $h_i(P)$}{\label{S:Parity}}
In this section we derive a series of different conditions under which the sum of the $h_i(P)$ must be even.
All of these conditions involve some $h_k(\partial \Gamma (P))$ and some $h_j(P)$ being zero.
In the case where $P$ has even dimension (so $d$ is odd) we already know from Theorem \ref{T:BasicCond} that if any $h_k(\partial \Gamma (P))$ is zero then $\sum_{i=0}^d h_i (P)$ is even, so these new conditions are only of interest for odd dimensional posets.

\subsection{Conditions from Counting Arguments}

Our first two examples of this new type of condition follow from counting arguments involving the faces of our complexes.
The main idea in both proofs is the following connection between a zero in the $h$-vector of the boundary of the complex and a parity condition on the incidences between facets and codimension-one faces.

\begin{lemma}{\label{L:EvenIncidence}}
Let $P$ be a $(d-1)$-dimensional simplicial poset such that $|\Gamma(P)|$ is a ball and $h_k(\partial \Gamma (P)) =0$ for some $k$ strictly between zero and $d-1$.
Then every set of $d-1$ vertices of $P$ is contained in an even number of facets (possibly zero).
\end{lemma}

\begin{proof}
Let $S$ be a set of $d-1$ vertices of $\Gamma(P)$.
If $S$ is contained in no facets of $\Gamma(P)$ we are done.
Otherwise, let $F$ be a face of $\Gamma (P)$ with vertex set $S$.
If $F$ is an interior face of $\Gamma(P)$ then since $|\Gamma(P)|$ is a manifold there must be exactly two facets of $\Gamma(P)$ that have $F$ as a codimension-one face.
If $F$ is a boundary face of $\Gamma(P)$ then there is exactly one facet of $\Gamma(P)$ that has $F$ as a codimension-one face.
Further, since some $h_k(\partial \Gamma (P)) =0$, by Theorem \ref{T:SphereZero} the number of boundary faces of $\Gamma(P)$ with vertex set $S$ is even.
Since no single facet of $\Gamma(P)$ can have multiple faces with the same vertex set, the total number of facets of $\Gamma (P)$ that contain $S$ is even.
\end{proof}

Now consider the case where $|\Gamma(P)|$ is a ball, $h_1 (P) = 0$, and $h_k(\partial \Gamma (P)) =0$ for some $k$ strictly between zero and $d-1$.
By Lemma \ref{L:EvenIncidence}, any set of $d-1$ vertices of $P$ is contained in an even number of facets.
In terms of the face numbers, $h_1(P)=0$ implies $f_0(P) = d$, meaning that all of the vertices of $\Gamma (P)$ are in every facet.
Therefore, every set of $d-1$ vertices of $P$ is in every facet.
Hence the total number of facets of $\Gamma(P)$ is even and we have the following proposition.

\begin{prop}{\label{T:h1=0}}
Let $P$ be a $(d-1)$-dimensional simplicial poset such that $|\Gamma(P)|$ is a ball, $h_1 (P) = 0$, and $h_k(\partial \Gamma (P)) =0$ for some $k$ strictly between zero and $d-1$.
Then $\sum_{i=0}^d h_i(P)$ is even.
\end{prop}

We can extend the result of Proposition \ref{T:h1=0} to the case $h_2(P)=0$ (instead of $h_1(P)=0$) using a somewhat more involved argument based on the same ideas.

\begin{prop}{\label{T:h2=0}}
Let $P$ be a $(d-1)$-dimensional simplicial poset such that $|\Gamma(P)|$ is a ball, $h_2 (P) = 0$, and $h_k(\partial \Gamma (P)) =0$ for some $k$ strictly between zero and $d-1$.
Then $\sum_{i=0}^d h_i(P)$ is even.
\end{prop}

\begin{proof}
Pick a facet $F_0$ of $\Gamma (P)$ with vertex set $V = \{v_1, \ldots , v_d \}$.
Let $\Delta_0$ be the induced subcomplex on the vertex set $V$; $\Delta_0$ consists of all faces of $\Gamma (P)$ who vertices are contained in the set $V$.
Note that $\Delta_0$ contains at least $\binom{d}{2}$ edges.

Let $F_1$ be a facet of $\Gamma (P) - \Delta_0$ that intersects $\Delta_0$ in a face of dimension $d-2$.  
Since $|\Gamma (P)|$ is a manifold, unless $\Gamma (P) = \Delta_0$ such a facet $F_1$ must exist.
Let $w_1$ be the vertex of $F_1$ not in $V$ and let $\Delta_1$ be the induced subcomplex of $\Gamma(P)$ on the vertex set $V \cup \{w_1\}$.
Note that there must be at least $d-1$ edges in $\Delta_1 - \Delta_0$ in order for the facet $F_1$ to exist.

We can continue to build our complex in this manner until we reach $\Delta_{h_1} = \Gamma(P)$.
This results in a minimum of $\binom{d}{2} + h_1(P) \cdot (d-1)$ edges in our complex.
However, since $h_2(P) = 0$ this is exactly the number of edges in $\Gamma(P)$.
So we must have added the minimum number of edges at each step in our construction.
In particular, for $1 \leq i \leq h_1$, all of the facets of $\Delta_i$ that contain $w_i$ must have the same vertex set as $F_i$.

By Lemma \ref{L:EvenIncidence} we know that every set of $d-1$ vertices of $\Gamma (P)$ is contained in an even number of facets.
In particular, let $S$ be a set of $d-1$ vertices of $F_{h_1}$ that includes the vertex $w_{h_1}$.
The facets that contain the vertices of $S$ are exactly those facets whose vertex set equals the vertex set of $F_{h_1}$.
Therefore, there must be an even number of facets on the vertex set of $F_{h_1}$.

Since we are only interested in the parity of the number of facets on each vertex set we can now ignore the contribution of the facets on the vertex set of $F_{h_1}$ and repeat the above argument on the complex $\Delta_{h_1-1}$ and the facet $F_{h_1-1}$.
Continuing in this manner we see that there are an even number of facets on all of the sets of $d$ vertices of $\Gamma (P)$.
Therefore $\Gamma(P)$ has an even number of facets, as desired.
\end{proof}

\subsection{The Cone Over the Boundary of $\Gamma(P)$}

Let $P$ be a simplicial poset such that $|\Gamma(P)|$ is a manifold with boundary and let $Q$ be the face poset of $\partial(\Gamma(P))$.
Define the cone over the boundary of $\Gamma (P)$ to be $SP := P \cup (Q \times [1,2])$ with each element $(p,1) \in (Q \times [1,2])$ identified with the element $p \in P$.
The covering relations in $SP$ are all of the covering relations in $P$ along with all of the covering relations in $Q \times [1,2]$.
In the case where $|\Gamma (P)|$ is a $(d-1)$-ball, $|\Gamma(SP)|$ is a $(d-1)$-sphere.

The face numbers of the new complex $\Gamma(SP)$ are given by
\[ f_i(\Gamma(SP)) = f_i (\Gamma(P)) + f_{i-1} (\partial \Gamma (P)) \]
for $-1 \leq i \leq d-1$, where $f_{-2}(\partial \Gamma(P))$ is interpreted as zero.
A straightforward calculation then shows that the elements of the $h$-vector of $\Gamma(SP)$ are given by
\begin{equation}{\label{E:SPh}}
h_i(\Gamma(SP)) 
= h_i (\Gamma(P)) + h_{i-1} (\partial \Gamma (P))
= h_i (\Gamma(P)) + \sum_{j=0}^{i-1} (h_j(\Gamma (P)) - h_{d-j}(\Gamma(P))),
\end{equation}
with the last equality by equation (\ref{E:Boundary}).

We now consider the relationship between the algebras $A_P$ and $A_{SP}$.
Let $v$ be the cone point of $SP$; $v$ is the vertex corresponding to $(\hat{0},2)$ in $Q \times [1,2]$.
There is a natural map $f:A_{SP} \rightarrow A_P$ given by setting all variables corresponding to faces containing $v$ equal to zero.
If $\Theta = \theta_1, \ldots, \theta_d$ is a l.s.o.p. for $A_{SP}$, then by Lemma \ref{L:lsop} the image of $\Theta$ under $f$ (which we also write as $\Theta)$ is a l.s.o.p. for $A_P$.
Therefore, there is an induced map $f: A_{SP}/ \Theta \rightarrow A_P / \Theta$ with kernel generated (modulo $\Theta$) by monomials containing a variable corresponding to a face containing $v$.
We use this map $f$ to prove the following lemma.

\begin{lemma}{\label{T:Interior}}
 Let $P$ be a $(d-1)$-dimensional simplicial poset such that $|\Gamma(P)|$ is a ball and $h_k(P)=0$ for some $k$ strictly between zero and $d$.
Let $F$ be a facet of $\Gamma(P)$ with an interior vertex.
Then $\Gamma(P)$ has an even number of facets with vertex set equal to that of $F$.
\end{lemma}

\begin{proof}
Let $V = \{v_1, \ldots , v_{d-1}, v_d\}$ be the vertex set of $F$, where $v_d$ is an interior vertex of $\Gamma(P)$.
Let $m$ be the monomial in $(A_{SP})_k$ defined by $m = x_{v_1} \cdots x_{v_k}$.
Since the dimension of $(A_P/ \Theta)_k$ is $h_k(P) = 0$ we know that $m$ is in the kernel of the map $f : A_{SP}/ \Theta \rightarrow A_P/ \Theta$ defined above.
Therefore, in $A_{SP}/ \Theta$ we can write $m$ as a linear combination of monomials each containing a variable corresponding to a face containing $v$.
Since $v_d$ is an interior vertex of $\Gamma(P)$, $x_{v_d} m$ is zero in $A_{SP}/ \Theta$.
Thus by Proposition \ref{T:SphereZeroAlg} we know that there must be an even number of facets of $\Gamma(SP)$ with vertex set $V$.
Since the cone point $v$ is not in $V$, the facets of $\Gamma(SP)$ with vertex set $V$ are exactly the same as the facets of $\Gamma(P)$ with vertex set $V$, proving the desired result.
\end{proof}

\subsection{The Case $h_1 (\partial \Gamma (P)) = 0$}
Using the algebraic framework developed in the previous section we can now prove the following necessary condition.

\begin{prop}{\label{T:bh1=0}}
Let $P$ be a $(d-1)$-dimensional simplicial poset such that $|\Gamma(P)|$ is a ball, $h_1(\partial \Gamma (P)) =0$, and $h_k (P) = 0$ for some $0 < k < d$.
Then $\sum_{i=0}^d h_i(P)$ is even.
\end{prop}

\begin{proof}
Since $h_1(\partial \Gamma (P)) =0$ we know that $\partial \Gamma (P)$ has only $d-1$ vertices.
Therefore every facet of $\Gamma(P)$ contains an interior vertex.
By Lemma \ref{T:Interior} there are an even number of facets (possibly zero) on every set of $d$ vertices of $\Gamma(P)$.
Hence $\sum_{i=0}^d h_i(P)$, which is the total number of facets of $\Gamma(P)$, is even.
\end{proof}

\subsection{The Case $h_1 (\partial \Gamma (P)) = 1$}
A slightly more involved argument allows us to extend the result of Proposition \ref{T:bh1=0} to the case where $h_1 (\partial \Gamma (P)) = 1$ and some higher $h_j (\partial (\Gamma (P))$ is zero.

\begin{prop}{\label{T:bh1=1}}
Let $P$ be a $(d-1)$-dimensional simplicial poset such that $|\Gamma(P)|$ is a ball, $h_1(\partial \Gamma (P)) = 1$, $h_j(\partial \Gamma (P)) = 0$ for some $1<j<d-1$, and $h_k (P) = 0$ for some $1 \leq k < d$.
Then $\sum_{i=0}^d h_i(P)$ is even.
\end{prop}

\begin{proof}
The assumption $h_1(\partial \Gamma (P)) =1$ implies that $\partial \Gamma (P)$ has $d$ vertices.
Let $W = \{w_1, \ldots , w_d\}$ be the set of exterior vertices of $\Gamma (P)$.
Let $V$ be a set of $d$ vertices of $\Gamma(P)$.
If $V \neq W$ then $V$ contains some interior vertex, so by Lemma \ref{T:Interior} we know that there are an even number of facets of $\Gamma(P)$ with vertex set $V$.
In particular, given any set $S$ of $(d-1)$ vertices of $\Gamma(P)$ there are an even number (possibly zero) of facets with vertex set $V$ that contain $S$.

If there are no facets with vertex set $W$ then we are done, so assume $F$ is a facet with vertex set $W$.
Let $W'$ be a set of $d-1$ distinct elements of $W$.
Since $h_j(\partial \Gamma (P)) = 0$ we know by Theorem \ref{T:SphereZero} that the number of boundary faces of $\Gamma(P)$ with vertex set $W'$ is even.
Because $|\Gamma(P)|$ is a manifold each interior face with vertex set $W'$ is contained in two facets of $\Gamma(P)$.
Therefore there are an even number of facets of $\Gamma(P)$ that contain the vertices $W'$.
As argued above there are an even number of facets on vertex sets other than $W$ that contain $W'$, so there must be an even number of facets on vertex set $W$.
Therefore we have an even total number of facets, which gives the desired result.
\end{proof}

\section{Constructions}{\label{S:Constructions}}

We now turn our attention to constructing posets $P$ with prescribed $h$-vectors such that $|\Gamma(P)|$ is a ball.
The balls that we construct are all shellable.
We use the following result of Bj{\"o}rner \cite[Proposition 4.3]{Bjorner84} to prove that the complexes that we construct are actually balls.

\begin{prop}{\label{T:ShellBall}}
Let $\Gamma (P)$ be a shellable CW-complex of dimension $d-1$.
If every $(d-2)$-cell is a face of at most two $(d-1)$-cells and some $(d-1)$-cell is a face of only one $d$-cell then $|\Gamma (P)|$ is homeomorphic to the ball $E^d$.
\end{prop}

The first theorem of this section presents our basic construction method.
The remainder of the section gives some extensions of this construction that allow us to obtain additional $h$-vectors.

\begin{thm}{\label{T:basicconstruction}}
Let $\mathbf{h}= (h_0,h_1, \ldots ,h_{d-1},h_d) \in \N^{d+1}$ with $h_0 =1$ and $h_d =0$.
Let $\partial h_j = \sum_{i=0}^j (h_i - h_{d-i})$.
If $\partial h_j > 0$ for $0 \leq j \leq \lfloor (d-1)/2 \rfloor$ then there exists a poset $P$ such that $|\Gamma(P)|$ is a $(d-1)$-ball and $h(P) = \mathbf{h}$.

Alternatively, let $0 < n < \lfloor (d-1)/2 \rfloor$ be the smallest number such that $\partial h_n = 0$.
If $\sum_{i=0}^d h_i$ is even and 
$\partial h_l \leq \sum_{i=0}^{n-1} h_{l-i}$ for $n+1 \leq l \leq d-(n+1)$
then there exists a poset $P$ such that $|\Gamma(P)|$ is a $(d-1)$-ball and $h(P) = \mathbf{h}$.
\end{thm}

Many of the conditions in Theorem \ref{T:basicconstruction} are related to the restrictions on the $h$-vectors of balls in Theorem \ref{T:BasicCond}.
Also note that the inequalities in the case $\partial h_1 = 0$ of Theorem \ref{T:basicconstruction} are known to be necessary by Theorem \ref{T:Monotone}.

\begin{proof}
Let $F_i$ be the $i$th facet in our shelling and $\Delta_j$ be the complex $\cup_{i=1}^j F_i$.
Each facet $F_i$ in our complex contains $d$ vertices.
We label these vertices $\{1\}_i, \{2\}_i, \ldots , \{d\}_i$.
For any set $S \subseteq [d]$ we denote by $\{S\}_i$ the face of $F_i$ containing exactly the vertices $\{ \{l\}_i\}_{l \in S}$.
For example $\{1,2,3,4\}_3$ is the face of $F_3$ containing the vertices $\{1\}_3, \{2\}_3, \{3\}_3$ and $\{4\}_3$.
We use the notation $\{a:b\}$ and $\{a:b\}_c$ to refer to $\{a,a+1, \ldots , b\}$ and $\{a,a+1, \ldots , b\}_c$ respectively.
For each facet $F_i$ we describe below the identifications of faces of $F_i$ with faces in $\Delta_{i-1}$.
Most of the vertices of $F_i$ will be identified with vertices of $\Delta_{i-1}$, but in some cases $F_i$ may contain a new vertex.
For example, we may state that $\{1\}_2$ is identified with $\{1\}_1$ or that $\{1\}_2$ is a new face.
In general, we choose the vertex labels such that two identified faces contain vertices labeled by the same numbers.

First we consider the case where all of the $\partial h_i$ are strictly positive.
Let $a = \sum_{i=0}^{d} h_i$, which is the total number of facets in our shelling.
For $1 \leq k \leq a-1$ let $c_k$ be the integer such that $\sum_{i=0}^{c_k-1} h_i < k +1 \leq \sum_{i=0}^{c_k} h_i$.
Thus $c_k$ measures the location where the sum of the entries of the vector $\mathbf{h}$ reaches $k+1$.
Set $c_0 = 0$.
Then $|\{ k : c_k = j\}| = h_j$.
As an example, if $\mathbf{h} = (1,2,0,0,1,0)$ then $a=4$,
$c_0 = 0$, $c_1 = 1$, $c_2 = 1$, and $c_3 = 4$.

We begin the shelling with the facet $F_1$ which cannot have any identifications with any previous facets,
hence $|\sigma (F_1)| = 0$.
The remaining facets will be added in pairs $F_i$, $F_{i+1}$ where $i$ is even.
The restriction faces of $F_i$ and $F_{i+1}$ will be $\{1:c_{i/2}\}_{i}$ and $\{c_{i/2}+1 : c_{i/2}+c_{a-i/2} \}_{i+1}$ respectively.
We are pairing a facet contributing to the start of the $h$-vector with a facet contributing to the end of the $h$-vector and then working our way inward to the center of the $h$-vector with the subsequent pairs of facets.
We stop after adding the facet $F_a$.

We now describe how the facets $F_i$ and $F_{i+1}$ are attached to our complex.
For $i$ even, we introduce a new face $\{1:c_{i/2}\}_{i}$.
Let $S \subseteq [d]$.
If $S \supseteq \{1:c_{i/2}\}$ then $\{S\}_i$ can not be identified with any face in a previous facet.
If $S \supseteq \{c_{i/2-1}+1 : c_{i/2}\}$ but $S \not\supseteq \{1 : c_{i/2-1}\}$ then identify $\{S\}_i$ with $\{S\}_{i-1}$.
For all other sets $S \subset [d]$ identify $\{S\}_i$ with $\{S\}_1$.
The fact that these identifications are well defined follows from the case $k = i$ of Claim \ref{C:consist} below.

Continuing the shelling, $F_{i+1}$ is identified with $F_i$ except we replace the face $\{c_{i/2}+1 : c_{i/2}+c_{a-i/2} \}_{i}$ by a new face $\{c_{i/2}+1 : c_{i/2}+c_{a-i/2} \}_{i+1}$ with the same boundary.
The fact that all of the $\partial h_i$ are positive ensures that $c_{i/2}+c_{a-i/2}$ never exceeds $d$, so the construction can proceed as described.

\begin{claim}{\label{C:consist}}
Fix $i$ even with $2 \leq i \leq a$ and $k$ even with $2 \leq k \leq i$.
Let $S \subseteq [d]$ such that 
$S \not\supseteq \{c_{i/2-1}+1 : c_{i/2}\}$ and $S \not\supseteq \{1 : c_{k/2-1}\}$.
Then $\{S\}_{k-1} = \{S\}_1$.
\end{claim}

\begin{proof}[Proof of Claim \ref{C:consist}]
Our proof is by induction on $i$.
The base case $i=2$ is trivial.
Assuming the result for $i = i'-2$ we prove the result for $i = i'$.
The inductive hypothesis allows us to assume that the construction of $\Delta_{i'-1}$ is well defined.

We prove the case $i = i'$ by induction on $k$.
Again, the base case $k=2$ is trivial.
We now assume the result for $k=k'-2$ and prove the result for $k = k'$.

We first show that $\{S\}_{k'-1} = \{S\}_{k'-2}$.
By the construction of the odd index facets, this follows from showing $S \not\supseteq \{c_{(k'-2)/2}+1 : c_{(k'-2)/2}+c_{a-(k'-2)/2} \}$.
By assumption $S \not\supseteq \{c_{i'/2-1}+1 : c_{i'/2}\}$.
Thus it is enough to show 
\begin{equation}{\label{E:consist}}
c_{(k'-2)/2}+1 \leq c_{i'/2-1}+1 \qquad \mbox{and} \qquad c_{i'/2} \leq c_{(k'-2)/2}+c_{a-(k'-2)/2}.
\end{equation}
The first inequality in (\ref{E:consist}) follows from the monotonicity of the $c_l$.
Again using the monotonicity of the $c_l$ and the fact that $i' \leq a$ we have
\[c_{i'/2} \leq c_{a-i'/2}  \leq c_{a-(k'-2)/2} \leq c_{(k'-2)/2}+c_{a-(k'-2)/2}, \]
proving the second inequality.

We complete the proof of Claim \ref{C:consist} by showing $\{S\}_{k'-2} = \{S\}_1$.
By assumption, $S \not\supseteq \{1:c_{k'/2-1}\} = \{1:c_{((k'-2)/2}\}$.
Thus, if $S \supseteq \{c_{(k'-2)/2-1}+1: c_{(k'-2)/2}\}$ then $S \not\supseteq \{1: c_{(k'-2)/2-1}\}$.
In this case, by our construction of the even index facets we have $\{S\}_{k'-2} = \{S\}_{k'-3}$ and by the inductive hypothesis $\{S\}_{k'-3} = \{S\}_1$, giving the desired result.
If $S \not\supseteq \{c_{(k'-2)/2-1}+1 : c_{(k'-2)/2}\}$ then our construction identifies $\{S\}_{k'-2}$ and $\{S\}_1$, completing the proof.
\end{proof}

Let $\{\hat{\jmath}\}_k$ be the codimension-one face of $F_k$ that does not contain the vertex $\{j\}_k$.
\begin{claim}{\label{C:intersect}}
Let $i \geq 2$ be even.
Then for $2 \leq k \leq a$, $F_k \cap \Delta_{k-1}$ is given by
\[ F_i \cap \Delta_{i-1} = \cup_{j=1}^{c_{i/2}} \{ \hat{\jmath} \}_i
 \qquad \mbox{or} \qquad
F_{i+1} \cap \Delta_i = \cup_{j=c_{i/2}+1}^{c_{i/2}+c_{a-i/2}} \{ \hat{\jmath} \}_{i+1}.
\]
Hence the $F_i$ form a shelling order with $|\sigma (F_i)| = c_{i/2}$ and $|\sigma (F_{i+1})| = c_{a-i/2}$.
\end{claim}

\begin{proof}[Proof of Claim \ref{C:intersect}]
For $1 \leq j \leq c_{i/2-1}$ each of the faces $\{ \hat{\jmath} \}_i$ is also in $F_{i-1}$ while for $c_{i/2-1} < j \leq c_{i/2}$ the face $\{ \hat{\jmath} \}_i$ is also a face of $F_1$.
Therefore $\cup_{j=1}^{c_{i/2}} \{ \hat{\jmath} \}_i \subseteq (F_i \cap \Delta_{i-1})$.
To see the reverse inclusion note that any face in $F_i \backslash \left( \cup_{j=1}^{c_{i/2}} \{ \hat{\jmath} \}_i \right)$ contains the face $\{1:c_{i/2}\}_i$, which is a new face and therefore not in $\Delta_{i-1}$.

The proof for $F_{i+1}$ is handled in a similar manner with $\{ \hat{\jmath} \}_{i+1} \subseteq F_i$ for $c_{i/2}+1 \leq j \leq c_{i/2}+c_{a-i/2}$ and $\{c_{i/2}+1:c_{i/2}+c_{a-i/2}\}_{i+1} \not\in \Delta_{i}$.
The last part of Claim \ref{C:intersect} now follows from the definition of a shelling.
\end{proof}

\begin{claim}{\label{C:ball}}
For $1 \leq p \leq a$ each codimension-one face of $\Delta_p$ is contained in at most two facets and there exists a codimension-one face of $\Delta_p$ that is contained in only one facet.
\end{claim}

\begin{proof}[Proof of Claim \ref{C:ball}]
For each \emph{new} codimension-one face $\{\hat{\jmath}\}_i$ of $F_i$ we describe the single facet $F_k$ with $k >i$ that may also contain $\{\hat{\jmath}\}_i$.
When $i=1$, $\{\hat{\jmath}\}_1$ is a face of $F_k$ where $k \leq a$ is the smallest even integer such that $c_{k/2} \geq j$ (if such a $k$ exists).
For $i>1$, $\{\hat{\jmath}\}_i$ may be a face of $F_{i+1}$.

Using Claim \ref{C:intersect}, it is straightforward to check that each codimension-one face in any $F_l \cap \Delta_{l-1}$ is described by one of the above cases.
Therefore, any codimension-one face is contained in at most two facets. 
Also note that the facet $\{\hat{d}\}_1$ is only contained in $F_1$, completing the proof of Claim \ref{C:ball}.
\end{proof}

By Claim \ref{C:ball} and Proposition \ref{T:ShellBall} we know that $\Delta_{a}$ is a ball.
Further, using the $|\sigma (F_i)|$ from Claim \ref{C:intersect} to count the contribution of each facet to the $h$-vector and the fact $|\{ i : c_i = j\}| = h_j$ we know that $\Delta_{a}$ has the desired $h$-vector.

We now consider the case where $0 < n < \lfloor (d-1)/2 \rfloor$ is the smallest integer such that $\partial h_n = 0$.
Define a new vector $\mathbf{h'}= (h'_0,h'_1, \ldots ,h'_{d-1},h'_d)$ by $h'_i = h_i$ for $i \neq d-n$ and $h'_{d-n} = h_{d-n}-1$.
From the definition of the $\partial h_i$, in order for $\partial h_{n-1} > 0$ and $\partial h_n =0$ we must have $h_{d-n} >0$, so the vector ${\bf h'}$ has non-negative entries.
Additionally, for $n \leq i \leq \lfloor (d-1)/2 \rfloor$ we have $\partial h'_{i} = \partial h_i+1$ ensuring that all of the $\partial h'_i$ are strictly positive.
We can therefore apply the construction of the previous case to create a ball with $h$-vector ${\bf h'}$.
In what follows we take $a = \sum_{i=0}^{d} h'_{i} = (\sum_{i=0}^{d} h_{i}) - 1$ to match the definition of $a$ in the previous case.

By assumption $\sum_{i=0}^d h_i$ is even, hence the construction of the ball with $h$-vector ${\bf h'}$ ends with the facet $F_a$ with $a$ odd.
We complete the construction of a ball with $h$-vector ${\bf h}$ by adding a facet $F_{a+1}$.
We attach $F_{a+1}$ to the ball $\Delta_{a}$ using the above rules for attaching an even index facet but acting as if $c_{(a+1)/2} = d-n$, so $\{1:d-n\}_{a+1}$ is a new face.
To complete the proof we must extend the results of Claims \ref{C:consist}, \ref{C:intersect}, and \ref{C:ball} to include the additional facet $F_{a+1}$.

First we prove Claim \ref{C:consist} for the case $i=a+1$.
The proof is the same as for smaller $i$ values except that the second inequality in (\ref{E:consist}) requires a different justification.
Rewriting this inequality, for $1 \leq j \leq (a-1)/2$ we must show 
\begin{equation*}
d-n \leq c_{j}+c_{a-j}.
\end{equation*}

First consider the case $c_j \geq d - 2n$.
Since $n < \lfloor (d-1)/2 \rfloor$, adding the equations $\sum_{m=0}^d h_m = a+1$ and $\sum_{m=0}^n(h_m-h_{d-m})=0$ and then removing some non-negative terms from the left-hand side yields $\sum_{m=0}^n h_m \leq (a+1)/2$.
Hence $\sum_{m=0}^n h'_m \leq (a+1)/2$ and $c_{(a+1)/2} \geq n$.
Since $a-j \geq (a+1)/2$ we have $c_{a-j} \geq c_{(a+1)/2} \geq n$, completing the proof of this case.

For the case $1 \leq c_j \leq d - 2n -1$ note that we can rewrite the assumption $\partial h_l \leq \sum_{m=0}^{n-1} h_{l-m}$ as 
\begin{equation*}
 \sum_{m=0}^{l-n} h_m \leq \sum_{m=0}^l {h_{d-m}}
\qquad \mbox{ or } \qquad
 \sum_{m=1}^{l-n} h'_m \leq \sum_{m=0}^l {h'_{d-m}}
\end{equation*}
where $ n+1 \leq l \leq d-(n+1)$.
By the second inequality, choosing $l$ such that $c_j = l -n$ we have $c_{a-j} \geq d-l$.
Therefore $c_j+c_{a-j} \geq d-n$, as desired.

We extend Claim \ref{C:intersect} by showing
\[ F_{a+1} \cap \Delta_{a} = \cup_{j=1}^{d-n} \{ \hat{\jmath} \}_{a+1}. \]
This follows from the proof of the even case of Claim \ref{C:intersect} by treating $c_{(a+1)/2} = d-n$.

The only change needed in the proof of Claim \ref{C:ball} to allow $p = a+1$ is noting that the faces $\{\hat{\jmath} \}_1$ for $c_{(a-1)/2} < j \leq d-n$ are also faces of $F_{a+1}$. 
\end{proof}

We next present a slight augmentation of the previous theorem that allows us to deal with some additional cases involving $h$-vectors that have a single sequence of non-zero entries.

\begin{thm}{\label{T:NonZeroConstruction}}
Let $\mathbf{h}= (h_0,h_1, \ldots ,h_{d-1},h_d) \in \N^{d+1}$ with $h_0 =1$.
Assume there exists $k \in \{1,2, \ldots, d-1\}$ such that $h_j = 0$ for $j>k$ and $h_j > 0$ for $1 \leq j \leq k$.
Define $\mathbf{h'} = (1,h_1-1,h_2-1, \ldots ,h_{k}-1,0, \ldots, 0)$.
If $\mathbf{h'}$ satisfies the conditions of Theorem \ref{T:basicconstruction} then there exists a poset $P$ such that $|\Gamma(P)|$ is a $(d-1)$-ball and $h(P) = \mathbf{h}$.
\end{thm}

\begin{proof}
We once again construct a shellable CW-complex with the desired $h$-vector.
We begin our construction with a series of facets $\{G_i\}_{i=0}^k$.
Let $\Delta_j = \cup_{i=0}^j G_i$.
The complex $\Delta_k$ will be a \emph{simplicial complex} on vertex set $[d+1]$.
We can therefore think of the faces of the $\Delta_k$ as subsets of $[d+1]$ as well as topological simplexes.

For $0 \leq i \leq k$ define $G_i = [d+1] - \{i+1\}$.
Then for $1 \leq i \leq k$
\[ G_i \cap \Delta_{i-1} = \cup_{j=1}^{i} \left( G_i -\{j\} \right). \]
Hence $\sigma(G_i) = \{1,2, \ldots, i\}$ and $|\sigma(G_i)| = i$.

We now complete the proof by performing the construction of Theorem \ref{T:basicconstruction} on the vector $\mathbf{h'}$ with a few alterations.
We omit the initial facet $F_1$ in Theorem \ref{T:basicconstruction}.
Instead we attach all of our additional facets to the boundary of $\Delta_k$.
In our new construction we replace the vertices $\{j\}_1$, $1 \leq j \leq d$, from Theorem \ref{T:basicconstruction} with the vertices $[d+1] \backslash \{k+2\}$ of $\Delta_k$, identifying vertices in the order preserving way.
We then replace the faces of $\partial F_1$ from Theorem \ref{T:basicconstruction} with the faces defined by the corresponding sets of vertices of $\Delta_k$.

We claim that performing the construction of Theorem \ref{T:basicconstruction} with this alteration gives a shellable ball, with a shelling order given by concatenating the order $G_1, G_2, \ldots, G_k$ with the order given by Theorem \ref{T:basicconstruction}.
To prove this we take every facet of $\partial F_1$ that is contained in a later facet in the construction of Theorem \ref{T:basicconstruction} and show that the corresponding $(d-1)$-subset of $[d+1] \backslash \{k+2\}$ is a face of $\partial \Delta_k$.
Let $H = [d+1] \backslash \{j,k+2\}$ be a $(d-1)$-subset of $[d+1] \backslash \{k+2\}$.
For $1 \leq j \leq k+1$, the only facet of $\Delta_k$ that contains $H$ is $G_{j-1}$, so $H$ is in $ \partial \Delta_k$.
For $k+3 \leq j \leq d+1$, the facet of $\partial F_1$ corresponding to $H$ will never be used in in the construction of Theorem \ref{T:basicconstruction} since $h_l = 0$ for $l>k$.

Totaling the contributions of all of the $|\sigma(G_i)|$ shows that the ball created in this manner has the desired $h$-vector.
\end{proof}

We now present one final construction specific to dimension five.
This construction allows us to complete the characterization of the possible $h$-vectors of five-balls in the following section.

\begin{prop}{\label{T:Dim5Construction}}
Let $\mathbf{h}= (h_0,h_1,h_2,h_3,h_4,0,0) \in \N^{7}$ with $h_0 =1$, $h_1,h_2 \neq 0$, $\sum_{i=0}^4 h_i$ odd, and $\partial h_2 = 1+ h_1 +h_2 -h_4 = 0$.
Then there exists a poset $P$ such that $|\Gamma(P)|$ is a five-ball and $h(P) = \mathbf{h}$.
\end{prop}

\begin{proof}
We begin this construction by using Theorem \ref{T:basicconstruction} to create a 5-ball with $h$-vector $(1,0,0,h_3-1,0,0,0)$.
Note that since $1+ h_1 +h_2 -h_4 = 0$, the parity of $h_3$ is the same as the parity of $\sum_{i=0}^4 h_i$, which we assumed to be odd.
Therefore $h_3-1$ is even and non-negative.
When we complete this construction the facets of the boundary of our ball are
\begin{eqnarray*}
\{1,2,3,4,5 \}_1,  \qquad
\{1,2,3,4,6 \}_1,  \qquad
\{1,2,3,5,6 \}_1,  \qquad \\
\{2,3,4,5,6 \}_{h_3},  \qquad
\{1,3,4,5,6 \}_{h_3},  \qquad \mbox{and} \qquad
\{1,2,4,5,6 \}_{h_3}.
\end{eqnarray*}
Further, all of the faces of $F_{h_3}$ are identified with the corresponding faces of $F_1$ except for $\{1,2,3\}_{h_3} = \{1,2,3\}_{h_3-1}$, $\{4,5,6\}_{h_3}$, and all faces containing one of these two faces (there is also the easier case where $h_3 = 1$ and we have only one facet in this initial part of the shelling).
 
We now describe the next six facets of our shelling, altering our notation slightly to make the description easier to follow.

\noindent
$F_{h_3+1} = \{1,2,3,4,5,7\}_{h_3+1}$ where $\{7\}_{h_3+1}$ is a new vertex.

$\{1,2,3,4,5\}_{h_3+1}$ is identified with $\{1,2,3,4,5\}_1$.

Hence $\sigma(F_{h_3+1}) = \{7\}_{h_3+1}$.

\noindent
$F_{h_3+2} = \{1,2,3,4,6,7\}_{h_3+2}$.

$\{1,2,3,4,6\}_{h_3+2}$ is identified with $\{1,2,3,4,6\}_1$.

$\{1,2,3,4,7\}_{h_3+2}$ is identified with $\{1,2,3,4,7\}_{h_3+1}$.

Hence $\sigma(F_{h_3+2}) = \{6,7\}_{h_3+2}$.

\noindent
$F_{h_3+3} = \{1,2,3,5,6,7\}_{h_3+3}$.

$\{1,2,3,5,6\}_{h_3+3}$ is identified with $\{1,2,3,5,6\}_1$.

$\{1,2,3,5,7\}_{h_3+3}$ is identified with $\{1,2,3,5,7\}_{h_3+1}$.

$\{1,2,3,6,7\}_{h_3+3}$ is identified with $\{1,2,3,6,7\}_{h_3+2}$.

Hence $\sigma(F_{h_3+3}) = \{5,6,7\}_{h_3+3}$.

\noindent
$F_{h_3+4} = \{1,2,4,5,6,7\}_{h_3+4}$ with $\{4,5,6\}_{h_3+4} = \{4,5,6\}_{h_3}$.

$\{1,2,4,5,6\}_{h_3+4}$ is identified with $\{1,2,4,5,6\}_{h_3}$.

$\{1,2,4,5,7\}_{h_3+4}$ is identified with $\{1,2,4,5,7\}_{h_3+1}$.

$\{1,2,4,6,7\}_{h_3+4}$ is identified with $\{1,2,4,6,7\}_{h_3+2}$.

$\{1,2,5,6,7\}_{h_3+4}$ is identified with $\{1,2,5,6,7\}_{h_3+3}$.

Hence $\sigma(F_{h_3+4}) = \{4,5,6,7\}_{h_3+4}$.

\noindent
$F_{h_3+5} = \{1,3,4,5,6,7\}_{h_3+5}$ with $\{4,5,6\}_{h_3+5} = \{4,5,6\}_{h_3}$ and new face $\{4,5,6,7\}_{h_3+5}$.

$\{1,3,4,5,6\}_{h_3+5}$ is identified with $\{1,3,4,5,6\}_{h_3}$.

$\{1,3,4,5,7\}_{h_3+5}$ is identified with $\{1,3,4,5,7\}_{h_3+1}$.

$\{1,3,4,6,7\}_{h_3+5}$ is identified with $\{1,3,4,6,7\}_{h_3+2}$.

$\{1,3,5,6,7\}_{h_3+5}$ is identified with $\{1,3,5,6,7\}_{h_3+3}$.

Hence $\sigma(F_{h_3+5}) = \{4,5,6,7\}_{h_3+5}$.

\noindent
$F_{h_3+6} = \{2,3,4,5,6,7\}_{h_3+6}$ with $\{4,5,6\}_{h_3+6} = \{4,5,6\}_{h_3}$ and new face $\{4,5,6,7\}_{h_3+6}$.

$\{2,3,4,5,6\}_{h_3+6}$ is identified with $\{2,3,4,5,6\}_{h_3}$.

$\{2,3,4,5,7\}_{h_3+6}$ is identified with $\{2,3,4,5,7\}_{h_3+1}$.

$\{2,3,4,6,7\}_{h_3+6}$ is identified with $\{2,3,4,6,7\}_{h_3+2}$.

$\{2,3,5,6,7\}_{h_3+6}$ is identified with $\{2,3,5,6,7\}_{h_3+3}$.

Hence $\sigma(F_{h_3+6}) = \{4,5,6,7\}_{h_3+6}$.

\noindent
Examining the $|\sigma(F_i)|$ shows that the ball we have constructed has $h$-vector  $(1,1,1,h_3,3,0,0)$.

We finish our construction using a slightly altered version of the construction of Theorem \ref{T:basicconstruction} on the vector $(1,h_1-1,h_2-1,0,h_4-3,0,0)$.
In place of the initial facet from Theorem \ref{T:basicconstruction} we use the final facet $F_{h_3+6}$ from the above construction (with the order preserving identification of the two facets' vertices).
Note that both $\{3,4,5,6,7\}_{h_3+6}$ and $\{2,4,5,6,7\}_{h_3+6}$ are on the boundary of our above constructed ball.
These are the only two codimension-one faces of $F_{h_3+6}$ that are used in performing the Theorem \ref{T:basicconstruction} construction on the vector $(1,h_1-1,h_2-1,0,h_4-3,0,0)$.
So we can finish the shelling in this manner which results in a ball with the desired $h$-vector.
\end{proof}

\section{A Summary of Known Conditions}{\label{S:Summary}}

Using the results of the previous two sections we now fully characterize all of the $h$-vectors of simplicial posets that are balls up through dimension six.

\begin{prop}[Dimension 3]{\label{T:dim3}}
Let $\mathbf{h}= (1,h_1,h_2,h_3,0) \in \Z^{5}$.
Then there exists a simplicial poset $P$ such that $|\Gamma(P)|$ is a three-ball and $\mathbf{h} = h(P)$ if and only if the following all hold.
\begin{enumerate}
 \item $h_i \geq 0$ for $1 \leq i \leq 3$.
 \item $h_3 \leq h_1+1$.
 \item If $h_1 =0$ and $h_3=1$ then $h_2$ is even.
\end{enumerate}
\end{prop}

\begin{proof}
The necessity of the first two conditions follows directly from Theorem \ref{T:BasicCond}, while the third condition is a consequence of Proposition \ref{T:h1=0}.

When $h_3 > h_1+1$ sufficiency follows from the first case of Theorem \ref{T:basicconstruction}.
If $h_3 = h_1+1$ and $h_2$ is even the second case of Theorem \ref{T:basicconstruction} gives the desired result.
Otherwise, $h_3 = h_1+1 > 1$ and $h_2$ is odd which means all of the $h_i$ for $0 \leq i \leq 3$ are non-zero and we can apply Theorem \ref{T:NonZeroConstruction} to obtain the desired construction.
\end{proof}

\begin{prop}[Dimension 4]{\label{T:dim4}}
Let $\mathbf{h}= (1,h_1,h_2,h_3,h_4,0) \in \Z^{6}$.
Then there exists a simplicial poset $P$ such that $|\Gamma(P)|$ is a four-ball and $\mathbf{h} = h(P)$ if and only if the following all hold.
\begin{enumerate}
 \item $h_i \geq 0$ for $1 \leq i \leq 4$.
 \item $h_4 \leq h_1+1$.  If $h_4 = h_1 +1$ then $h_2-h_3$ is even.
 \item $h_3+h_4 \leq h_1 + h_2 + 1$.
\end{enumerate}
\end{prop}

\begin{proof}
Necessity follows directly from Theorem \ref{T:BasicCond}.
When the second and third conditions are satisfied with strict inequality then sufficiency is a result of the first case of Theorem \ref{T:basicconstruction}.
If $h_4 = h_1 +1$ the fact that $h_2-h_3$ is even allows us to apply the second case of Theorem \ref{T:basicconstruction} to obtain the desired construction (it is straightforward to check that the inequalities needed for this construction are all satisfied in this case).
If $h_4 < h_1+1$ and $h_3+h_4 = h_1 + h_2 + 1$ we can again apply the second case of Theorem \ref{T:basicconstruction}.
\end{proof}

\begin{prop}[Dimension 5]{\label{T:dim5}}
Let $\mathbf{h}= (h_0,h_1,h_2,h_3,h_4,h_5,h_6) \in \Z^{7}$ with $h_0=1$ and $h_6=0$.
Let $\partial h_j = \sum_{i=0}^j (h_i - h_{6-i})$ for $0 \leq j \leq 5$.
Then there exists a simplicial poset $P$ such that $|\Gamma(P)|$ is a five-ball and $\mathbf{h} = h(P)$ if and only if the following all hold.
\begin{enumerate}
\item $h_i \geq 0$ for $1 \leq i \leq 5$.
 \item $\partial h_1 \geq 0$. If $\partial h_1 =0$ then $h_i \geq \partial h_i$ for $0 \leq i \leq 5$.  If $\partial h_1 =0$ and $h_j = 0$ for some $1 \leq j \leq 4$ then $\sum_{i=0}^{5} h_i$ is even.
 \item $\partial h_2 \geq 0$.  If $\partial h_2 =0$ and $h_1=0$ or $h_2=0$ then $\sum_{i=0}^{5} h_i$ is even.
\end{enumerate}
\end{prop}

\begin{proof}
Then necessity of condition one and the first inequality in the other two conditions follow from Theorem \ref{T:BasicCond}.
The additional inequalities when $\partial h_1 =0$ come from Theorem \ref{T:Monotone}.
The remainder of condition two comes from Proposition \ref{T:bh1=0}.
For the third condition, the case $h_1 = 0$ follows from Proposition \ref{T:h1=0} while the case $h_2=0$ is a result of Proposition \ref{T:h2=0}.

For sufficiency, if there is strict inequality in the second and third conditions we use the first case of Theorem \ref{T:basicconstruction}.
If $\partial h_1 =0$ and $\sum_{i=0}^{5} h_i$ is even then we can apply the second case of Theorem \ref{T:basicconstruction}.
If $\partial h_1 =0$ and $\sum_{i=0}^{5} h_i$ is odd then all of the $h_i$ for $1 \leq i \leq d-1$ are non-zero, so we apply Theorem \ref{T:NonZeroConstruction}.
It is not hard to check that reducing the elements of the $h$-vector by one will preserve the inequalities $h_i \geq \partial h_i$ for $0 \leq i \leq 5$.
The case $h_3 \geq \partial h_3$ uses the fact that the sum of the $h_i$ is odd.

If $\partial h_2 =0$, the inequality in the second case of Theorem \ref{T:basicconstruction} is always trivially satisfied.
This gives the needed construction whenever $\sum_{i=0}^{5} h_i$ is even.
For the case where $\sum_{i=0}^{5} h_i$ is odd, when $h_i>0$ for $1 \leq i \leq 5$ we can apply Theorem \ref{T:NonZeroConstruction}.
Otherwise we use Proposition \ref{T:Dim5Construction}.
\end{proof}

\begin{prop}[Dimension 6]{\label{T:dim6}}
Let $\mathbf{h}= (h_0,h_1,h_2,h_3,h_4,h_5,h_6,h_7) \in \Z^{8}$ with $h_0=1$ and $h_7=0$.
Let $\partial h_j = \sum_{i=0}^j (h_i - h_{7-i})$ for $0 \leq j \leq 6$.
Then there exists a simplicial poset $P$ such that $|\Gamma(P)|$ is a six-ball and $\mathbf{h} = h(P)$ if and only if the following all hold.
\begin{enumerate}
 \item $h_i \geq 0$ for $1 \leq i \leq 5$.
 \item $\partial h_i \geq 0$ for $0 \leq i \leq 3$.
 \item If $\partial h_i = 0$ for any $1 \leq i \leq 3$ then $\sum_{i=0}^{6} h_i$ is even.
 \item If $\partial h_1 = 0$ then $h_i (P) \geq \partial h_i$ for all $i \geq 0$.
\end{enumerate}
\end{prop}

\begin{proof}
For the first three conditions necessity follows directly from Theorem \ref{T:BasicCond}, while for the last condition it is a result of Theorem \ref{T:Monotone}.
When the second condition is satisfied with all strict inequalities sufficiency is proved using the first case of Theorem \ref{T:basicconstruction}; otherwise we use the second case of Theorem \ref{T:basicconstruction}.
The fact that the inequalities of Theorem \ref{T:basicconstruction} are satisfied follows by assumption for the case where $\partial h_1 =0$ and an easy calculation for the case $\partial h_2 =0$.  The inequalities are vacuous for the case $\partial h_3 =0$.
\end{proof}

For even dimensional balls ($d$ odd), when any entry of the boundary $h$-vector is zero the sum of the $h_i$ of the ball must have even parity.
In the odd dimensional case, this relationship is lost and more subtle behavior occurs.
In particular, whether or not some of the $h_i$ of the ball are zero needs to be considered, resulting in some of the more complicated conditions and the extra construction in the five dimensional case.

When we move beyond dimension six, we no longer have a complete characterization of the possible $h$-vectors even for the even dimensional cases.
In particular, the inequalities in the second part of Theorem \ref{T:basicconstruction} are known to be necessary only for the case $l = 1$.
In higher dimensions, it is possible to create $h$-vectors that satisfy all of our known necessary conditions but violate these inequalities for $l > 1$.
For example, in dimension eight we do not know if the vector $(1,1,1,5,1,1,1,2,1,0)$, with corresponding boundary $h$-vector $(1,1,0,4,4,4,0,1,1)$, is the $h$-vector of a simplicial poset that is a ball.
For balls of odd dimension the situation is even less clear.
We still need a general framework to describe what types of conditions on the $h$-vectors of the balls and their boundary spheres require an even number of facets.
Additionally, as we saw in dimension five more constructions are needed to obtain all possible $h$-vectors.

\section*{Acknowledgments}
The author would like to thank Ed Swartz for his valuable advice and support.

\bibliographystyle{plain}
\bibliography{Ballsbib}

\end{document}